\documentclass[11pt]{article} 

\usepackage[utf8]{inputenc} 
  
\usepackage{geometry} 
\usepackage{graphicx} 

\usepackage{booktabs} 
\usepackage{array} 
\usepackage{paralist} 
\usepackage{verbatim} 
\usepackage{amsmath}
\usepackage{amssymb}
\usepackage{amsthm}
\usepackage{bbm}
\usepackage{xifthen}
\usepackage{xcolor}

\usepackage{subcaption}

\usepackage[colorlinks=true,linkcolor=blue,urlcolor=blue]{hyperref}

\usepackage{fancyhdr} 
\usepackage{sectsty}
\allsectionsfont{\sffamily\mdseries\upshape} 

\newtheorem{theorem}{Theorem}

\newtheorem{definition}[theorem]{Definition}

\newtheorem{proposition}[theorem]{Proposition}
\newtheorem{remark}[theorem]{Remark}

\usepackage[nottoc,notlof,notlot]{tocbibind} 
\usepackage[titles,subfigure]{tocloft} 

\newcommand{\E}{\mathbb E}

\newcommand{\F}{\mathcal F}

\renewcommand{\P}{\mathbb P}

\newcommand\cA{\mathcal{A}}

\newcommand\cD{\mathcal{D}}

\newcommand\cF{\mathcal{F}}

\newcommand\cL{\mathcal{L}}

\newcommand\bE{\mathbb{E}}
\newcommand\bF{\mathbb{F}}

\newcommand\bP{\mathbb{P}}

\newcommand{\ind}{\mathbbm{1}}

\newcommand{\kom}[1]{}
\renewcommand{\kom}[1]{{\bf [#1]}}

\newcounter{komcounter}
\numberwithin{komcounter}{section}

\title{The de Finetti problem with unknown competition}
\author{Erik Ekström\\\textit{Department of Mathematics, Uppsala University}\\ \\
Alessandro Milazzo\\\textit{Department of Mathematics, Uppsala University}\\\\
Marcus Olofsson\\\textit{Department of Mathematics and Mathematical Statistics, Umeå University}
}

\begin{document}

\maketitle

\begin{abstract}
	We consider a resource extraction problem which extends the classical de Finetti problem for a Wiener process to include 
	the case when a competitor, who is equipped with the possibility to extract all the remaining resources in one piece, may exist;
	we interpret this unknown competition as the agent being subject to possible fraud. 
	This situation is modelled as a controller-and-stopper non-zero-sum stochastic game with incomplete information.
	In order to allow the fraudster to hide his existence, we consider strategies where his action time is randomised. Under these conditions, 
	we provide a Nash equilibrium which is fully described in terms of the corresponding single-player de Finetti problem. In this equilibrium, the agent and the fraudster use singular strategies in such a way that a two-dimensional process, which represents available resources and the filtering estimate of active competition, reflects in a specific direction along a given boundary.
\end{abstract}

\section{Introduction}\label{intro}

In the classical single-player de Finetti problem for a Wiener process, the value of a limited resource evolves, in the absence of extraction, as 
\[Y_t=x + \mu t + \sigma W_t,\]
where $\mu$ and $\sigma$ are { positive} constants and $W$ is a standard Brownian motion. The de Finetti problem
-- also known as \textit{the dividend problem} -- then consists of maximising 
\[\E\left[ \int_0^{\tau_0}e^{-rt}dD_t\right]\]
over all {adapted, non-decreasing, and right-continuous} processes $D$ with $D_{0-}=0$, where $\tau_0:=\inf\{t\geq 0:Y_t-D_t\leq 0\}$ is the extinction time (or bankruptcy time). It is well-known (see, e.g., Asmussen and Taksar \cite{AT} and Jeanblanc and Shiryaev \cite{jeanblanc1995optimization}) that
the optimal strategy $\tilde D$ is given by $\tilde D_t=\sup_{0\leq s\leq t}(Y_s-B)^+$, where {$(x)^+:=\max\{x,0\}$ and} $B$ is a constant that can be calculated explicitly.

In the current article, we study the de Finetti problem under the threat of {unknown competition. We interpret this unknown competition as the agent, who exerts the control $D$ to extract from the source $Y$, being subject to possible fraud.}  
More precisely, we include the possibility that a fraudster exists, with the capacity to
extract all the remaining resources at once at a random time $\gamma$.
We use a Bernoulli random variable $\theta$ to model whether the fraudster exists ($\theta=1$) or not ($\theta=0$) and we consider maximisation of 
\[\E\left[ \int_0^{\tau_0\wedge\hat\gamma}e^{-rt}dD_t\right]\]
over controls $D$ {as above} and 
where $\hat\gamma:=\gamma 1_{\{\theta=1\}}+\infty 1_{\{\theta=0\}}$. At the same time, the fraudster seeks to choose 
$\gamma$ to optimise the expected payoff
\[\E\left[e^{-r(\tau_0\wedge\gamma)} X ^D_{\tau_0\wedge\gamma}\right],\]
where $X^D=Y-D$ represents the remaining resources after extraction.

The above game is a controller-and-stopper non-zero-sum stochastic game and we extend the stream of literature on stochastic games of control and stopping: Karatzas and Sudderth \cite{karatzas2006stochastic} studied three stochastic games of classical control and stopping for a linear diffusion. Karatzas and Zamfirescu \cite{karatzas2008martingale} developed a martingale approach for studying zero-sum stochastic games combining classical controls and stopping in a non-Markovian framework. Bayraktar and Huang \cite{bayraktar2013multidimensional} studied multidimensional controller-and-stopper zero-sum stochastic games in finite horizon. Kwon and Zhang \cite{kwon2015game} investigated a stochastic game combining singular control and stopping. Hernandez-Hernandez et al.~\cite{hernandez2015zero} studied a zero-sum game between a singular stochastic controller and a discretionary stopper. Bovo et al.~\cite{bovo2022variational} applied PDE methods to study variational inequalities on unbounded domains for zero-sum games between a singular stochastic controller and a discretionary stopper in finite horizon. De Angelis and Ferrari \cite{de2018stochastic} established a connection between a class of two-player non-zero-sum games of optimal stopping and certain two-player non-zero-sum games of singular control.
	
In contrast to most of the literature on stochastic games of control and stopping, which studies zero-sum games, we formulate and solve a \textit{non-zero-sum} game. Moreover, a relevant feature that distinguishes our game from the works mentioned above is \textit{incomplete information}. In our framework, incomplete information 
stems from the fact that the existence of the fraudster is uncertain. 
Since the fraudster is equipped with a binary stopping 
control, inference about the existence of the fraudster is based on observations of the events $\{\hat\gamma\leq t\}$. In fact, the strategies that we consider are based on observations/calculations of the two-dimensional process $(X,\Pi)$, where 
$X=X^D=Y-D$ is observed and represents the value of resources after extraction, and $\Pi$ it calculated, corresponding to the adjusted belief of active competition,
i.e., the conditional probability that $\theta=1$ given that stopping has not yet occurred, see Section~\ref{adjbelief}.

Remarkably, this controller-and-stopper non-zero-sum game with incomplete information has an equilibrium 
which can be described explicitly. 
The equilibrium is derived using the Ansatz that the equilibrium value for the controller is $(1-p)V(x)$, where 
$p$ is the initial probability of active competition, and $V$ is the value in the single-player de Finetti problem.
In this equilibrium the controller extracts resources 
and the fraudster stops at a randomised stopping time, specified in terms of a generalised intensity, in such a way that 
the corresponding two-dimensional process $(X,\Pi)$ reflects obliquely at a given monotone boundary $x=b(p)$. To construct
this two-dimensional reflected process, including a carefully specified reflection direction, we use the notion of perturbed Brownian motion (see, e.g., Carmona et al.~\cite{CPY} and Perman and Werner \cite{PW}).

Our paper is the third in a series of papers investigating the role of uncertain competition in stochastic games. This strand of research was initiated by De Angelis and Ekstr\"{o}m \cite{EkstromDeAngelis} in which the term ``ghost'' was also introduced to represent the players that may not exist. In \cite{EkstromDeAngelis} an optimal stopping game in which both players are uncertain of the existence of the opponent was studied. Next, Ekstr\"{o}m et al.~\cite{ELO} proposed and studied a 
ghost game in a setting related to fraud detection and so called ``salami slicing'' fraudulence. As in the current paper, 
a controller-and-stopper non-zero-sum game of ghost type is studied in \cite{ELO}, but with the ``ghost" role inverted. More precisely, in \cite{ELO} the controller is a ghost whereas
in the current paper the stopper is a ghost.
{In \cite{EkstromDeAngelis}, the ghost has also a stopping control and a similar Ansatz as above was shown to hold, namely,}
an equilibrium with equilibrium value $(1-p)V$ is obtained, where again $p$ is the probability of competition and $V$ is 
the value in the corresponding single-player game. Similar observations can be made also in non-dynamic auction games with unknown competition, see Hirshleifer and Riley \cite[pages 386-389]{HR}. On the other hand, in the setting of \cite{ELO} with a ghost controller, such an Ansatz was not used, but instead an equilibrium was obtained using variational methods. In view of this, a rule-of-thumb seems to be that 
the equilibrium value in the case of a ghost game where the ghost is equipped with a {\em stopping} control is given 
by $(1-p)V$, where $V$ is the value in the corresponding single-player game.
A precise formulation and verification of such a claim remains to be found.

The paper is organized as follows. In Section~\ref{setup} we provide the precise game formulation of the {de Finetti problem under unknown competition}. In Section~\ref{background} we review the standard single-player {de Finetti problem} and we 
provide properties of its game version that should hold in equilibrium using heuristic arguments. Section~\ref{perturbed} uses the notion of perturbed Brownian motion to construct the candidate equilibrium. Our main result Theorem~\ref{thm:verification}, in which the candidate equilibrium is verified, is presented in Section~\ref{verification}.
Finally, Section~\ref{num} illustrates our findings with a numerical study.

\section{Problem set-up}\label{setup}
We begin by setting the mathematical stage necessary for our analysis. Throughout the paper, we let $(\Omega, \mathcal F, \mathbb P)$ be a {complete} probability space on which a standard Brownian motion $W$, a Bernoulli random variable $\theta$ with $\P( \theta=1) =1-\P(\theta=0)=p\in[0,1]$ and a Uniform-$(0,1)$ random variable $U$ are defined. Moreover, $W$, $\theta$ and $U$ are assumed to be independent.

We consider a stochastic game between Player 1 and Player 2 in which both players seek to maximise certain quantities to be specified below. 
Let $Y$ be a Brownian motion with drift given by
$$
Y_t = x+\mu t + \sigma W_t ,
$$
where the initial condition satisfies $x\geq 0$ and $\mu$ and $\sigma$ are given positive constants.
Denote by $\mathbb F^W=(\F^W_t)_{0\leq t<\infty}$ the augmentation of the filtration generated by the Brownian motion $W$;
this filtration will represent the information that Player 1 (the ``controller'') is equipped with.

\begin{definition}[Admissible controls for Player 1]
An admissible control for Player 1 is a non-decreasing, right-continuous, $\bF^W$-adapted processes
$D=(D_t)_{t\geq 0}$  satisfying $D_{0-}=0$ and $D_t\leq Y_{t}$ for every $t\in[0,\infty)$.
We denote by $\mathcal A_1$ the set of admissible controls for Player 1.
\end{definition}

{For any strategy $D\in\cA_1$, let $X=X^D:=Y-D$ and define
\begin{equation}\label{tauD}
\tau_0^X:=\inf\{t\geq 0: X_t\leq 0\}.
\end{equation}
To simplify the notation, we will often omit the superscript and simply write $X$ instead of $X^D$ and $\tau_0$ instead of $\tau_0^X$.}

In order to let Player~2 (the ``fraudster") hide his existence, he will be equipped with randomized stopping times. To define the strategies of
Player~2, we denote by $\mathcal D$ the Skorokhod space of cadlag paths on $[0,\infty)$.

\begin{definition}[Admissible controls for Player 2]
An admissible control $\Gamma=(\Gamma_t(X))_{t\geq 0}$ for Player~2 is a mapping $(t,X)\mapsto \Gamma_t(X)$
from $[0-,\infty)\times \mathcal D$ into $[0,1]$ which is progressively measurable for the canonical
filtration on $\mathcal D$, non-decreasing and right-continuous in $t$, and satisfying $\Gamma_{0-}(X)=0$.
We denote by $\mathcal A_2$ the set of admissible controls for Player~2.
\end{definition}

Given a pair of admissible strategies $(D,\Gamma)\in\mathcal A_1\times\mathcal A_2$, we define a 
randomized stopping time $\gamma$ as
\begin{equation}\label{gamma}
\gamma:=\gamma^{\Gamma}:=\inf \{t\geq 0 : \Gamma_t(X^D) > U\},
\end{equation}
where we recall that $U$ is a random variable which is $\mbox{Unif(0,1)}$-distributed 
and independent of $\theta$ and $W$. {In accordance with the notation for $X=X^D$, we will often omit the superscript and simply write $\gamma$ instead of $\gamma^{\Gamma}$}.

\begin{remark}
We note that 
Player 2 selects a universal map $\Gamma$ that he will apply to any given path of $X=Y-D$ to generate his randomized stopping time $\gamma=\gamma^{\Gamma}$ in \eqref{gamma}. In this way, Player~2 is equipped with feed-back controls, and we will obtain a Markovian game structure.
\end{remark}

Given a fixed discount rate $r >0$ and a pair $(D,\Gamma) \in \mathcal A_1 \times  \mathcal A_2$, we define the payoffs for Player~1 and Player 2 as 
\begin{equation} 
\label{eq:valuefunctions}
J_1(x,p,D,\Gamma):=\mathbb E \left [ \int_{0}^{\tau_0 \wedge \hat \gamma} e^{-rt} dD_t \right]  
\end{equation}
and
\begin{equation}
\label{eq:valuefunctions2}
J_2(x,p,D,\Gamma) := \mathbb E \left [ e^{-r (\tau_0\wedge\gamma)} X_{\tau_0\wedge\gamma}\right ], 
\end{equation}
respectively, where 
$\tau_0=\tau_0^X$ and $\gamma=\gamma^{\Gamma}$ are defined as in \eqref{tauD}-\eqref{gamma}, and
$$
\hat \gamma:= \begin{cases}
\gamma  & \mbox{if $\theta=1$} \\
\infty & \mbox{if $\theta=0$}.
\end{cases}
$$
The integral in \eqref{eq:valuefunctions} is interpreted in the {Lebesgue-Stieltjes} sense,
with 
\[\int_{0}^{\tau_0 \wedge \hat \gamma} e^{-rt} dD_t :=\int_{[0,\tau_0 \wedge \hat \gamma]} e^{-rt} dD_t .\]
The inclusion of the lower limit $0$ of integration thus accounts for the contribution to Player~1 from an
initial push $dD_0=D_0 > 0$.

Each player seeks to maximise their respective profit, and we are looking for a Nash equilibrium to this non-zero-sum game in the sense of the following definition. 

\begin{definition}\label{DefNE} A pair $(D^*, \Gamma^*) \in \mathcal A_1 \times \mathcal A_2$ is a Nash equilibrium (NE) if
\begin{align*}
J_1 (x,p, D^*, \Gamma^*) &\geq J_1(x,p,D, \Gamma^*) \\
J_2 (x,p,  D^*, \Gamma^*) &\geq J_2(x,p, D^*, \Gamma) 
\end{align*}
for any pair  $(D, \Gamma)  \in \mathcal A_1 \times \mathcal A_2$.
\end{definition}

\begin{remark}
Note that it is a consequence of the game set-up that Player~1 has precedence {over} Player~2 in the sense 
that if a lump sum $dD_t>0$ is paid out at the same time $t=\hat\gamma$ as Player~2 stops, then 
Player~1 receives the lump sum, whereas Player~2 receives the reduced amount $Y_{t}-D_t$.
Consequently, since Player~1 may choose a strategy with $D_0 = x$, for any Nash equilibrium $(D^*, \Gamma^*) \in \mathcal A_1 \times \mathcal A_2$ we must have
\[J_1(x,p,D^*,\Gamma^*)\geq \sup_{D\in\mathcal A_1}J_1(x,p,D,\Gamma^*)\geq x.\]
\end{remark}

\begin{proposition}\label{alt}
For a given pair $(D,\Gamma)\in\mathcal A_1\times\mathcal A_2$, we have
\[J_1(x,p,D,\Gamma) = \mathbb E \left [ \int_{0}^{\tau_0 } e^{-rt} (1-p\Gamma_{t-}) dD_t \right],\]
where $\Gamma_t:=\Gamma_t(Y-D)$.
\end{proposition}

\begin{proof}
By conditioning we have
\begin{eqnarray*}
J_1(x,p,D,\Gamma) &=& \mathbb E \left [ \int_{0}^{\tau_0 \wedge \hat \gamma} e^{-rt} dD_t \right]= \mathbb E \left [ \int_{0}^{\tau_0 } e^{-rt} \ind_{\{t\leq \hat\gamma\}} dD_t \right]\\
&=& \mathbb E \left [ \int_{0}^{\tau_0 } e^{-rt} \P(t\leq \hat\gamma\vert \F^W_t) dD_t \right] .
\end{eqnarray*}
Since
$$\{\Gamma_{t-}< U\}\subseteq \{t\leq \gamma\}\subseteq \{\Gamma_{t-}\leq  U\},
$$
we have that 
\[\P(t\leq \hat\gamma\vert \F^W_t) = 1-p + p \P(t\leq \gamma\vert \F^W_t) = 
1-p\Gamma_{t-},\]
so 
\[J_1(x,p,D,\Gamma)= 
\mathbb E \left [ \int_{0}^{\tau_0 } e^{-rt} (1-p\Gamma_{t-}) dD_t \right].\]
\end{proof}
\begin{remark}
	Notice that for Player 2 we have chosen to maximise his expected payoff when he is active, i.e., when $\theta=1$. Alternatively, one could set Player 2 to maximise
	$$\hat{J}_2(x,p,D,\Gamma) := \mathbb E \left [\theta e^{-r (\tau_0\wedge\hat{\gamma})} X_{\tau_0\wedge\hat{\gamma}}\right ].$$
	The formulations for $J_2$ and $\hat{J}_2$ have the following interpretations. Imagine that before the game starts, at time $t=0-$, neither player knows $\theta$ and that the value of $\theta$ will be revealed to Player 2 at time $t=0$. Then, $\hat{J}_2$ is the expected payoff for Player 2 at time $t=0-$, whereas $J_2$ is the expected payoff at time $t=0$ when $\theta=1$. These games are referred to as the \textit{ex-ante} version of the game and the \textit{interim} version of the game, respectively (see \cite{aumann1995repeated, harsanyi1967games} for classical theory of games under incomplete information). Also notice that the two formulations are equivalent as by independence one obtains $\hat{J}_2(x,p,D,\Gamma)=pJ_2(x,p,D,\Gamma)$ and so the second inequality in Definition \ref{DefNE}  can be equivalently replaced by $\hat{J}_2 (x,p, D^*, \Gamma^*) \geq \hat{J}_2(x,p, D^*, \Gamma)$ for $p> 0$.
\end{remark}

\section{Background material and heuristics}\label{background}

\subsection{The single-player {de Finetti problem}}
Note that if $p=0$, then Player~1 acts under no competition and thus faces the standard {de Finetti} problem for which the value function
\begin{equation}\label{V}
V(x) := \sup_{D \in \mathcal A_1} \mathbb E \left[ \int_{0}^{\tau_0} e^{-rt} d D_t\right]
\end{equation}
and the optimal strategy $\tilde D$ are well known (see, e.g., \cite{jeanblanc1995optimization}). 
To describe this solution in more detail, let $\psi$ be the unique increasing solution of
\begin{equation*}
\cL \psi(x) =0, \quad x\geq 0,  
\end{equation*}
with $ \psi(0)=0$ and $ \psi'(0)=1,$
where $\cL$ denotes the differential operator
\begin{equation}\label{eq:L}
\cL:=\frac{\sigma^2}{2}\partial^2_x+\mu\partial_x-r.
\end{equation}
More explicitly, 
\begin{equation}
\label{eq:psi}\psi(x)= \frac{e^{\zeta_2 x}-e^{\zeta_1 x}}{\zeta_2-\zeta_1},
\end{equation}
where $\zeta_i$, $i=1,2$ are the solutions of the quadratic equation
\[\zeta^2+ \frac{2\mu}{\sigma^2}\zeta -\frac{2r}{\sigma^2}=0\]
with $\zeta_1<0<\zeta_2$.
Setting 
\begin{equation} \label{eq:B}
B:=\frac{\ln (\zeta_1^2) -  \ln(\zeta_2^2)}{\zeta_2-\zeta_1},
\end{equation}
we have that $\psi$ is concave on $[0,B]$ and convex on $(B,\infty)$, and 
\begin{align}
V(x) &= \begin{cases}
\frac{\psi(x)}{\psi'(B)}, & x\leq B, \\
x-B+V(B), & x>B.
\end{cases}\label{eq:Vtilde} 
\end{align}
Moreover,
\begin{equation}
\label{Dsingle}
\tilde D_t = \sup_{s\in[0,t]}\big(Y_s -B\big)^+
\end{equation}
is an optimal strategy in \eqref{V}, i.e.,
\[V(x)=\mathbb E \left[ \int_{0}^{\tilde{\tau}_0} e^{-rt} d \tilde D_t\right],\]
where $\tilde{X}:=X^{\tilde{D}}$ and $\tilde{\tau}_0:=\tau_0^{\tilde{X}}$. We also remark that $(\tilde{X}, \tilde D)$ is the solution of a Skorokhod reflection problem with reflection at the barrier $B$.

\subsection{Adjusted beliefs}\label{adjbelief}
We now return to our version of the game including  a ghost feature as described in Section~\ref{setup}. At the beginning of the game, from the perspective of Player 1 there is active competition (i.e., $\theta=1$) with probability $p$. As time passes, and if no stopping occurs, Player 1's conditional probability of competition $\Pi$ will decrease. 
More precisely, at time $t\geq0$, assuming that the strategy pair $(D,\Gamma)\in\mathcal A_1\times\mathcal A_2$ is played, 
we have 
\begin{eqnarray}\label{Pi}
\notag
\Pi_t =\Pi_t^{\Gamma}&:=& \P(\theta=1 |\mathcal \cF_t^W, \hat \gamma > t ) = \frac{\P(\theta=1 ,\hat \gamma > t | \mathcal F_t^W )}{\P(\hat \gamma>t| \mathcal F_t^W)} \\
&=&\frac{p \bP(\gamma > t | \mathcal F_t^W )}{(1-p) + p \P(\gamma>t| \mathcal F_t^W)}= \frac	{p (1-\Gamma_t(X^D))}{1-  p \Gamma_t(X^D)} 
\end{eqnarray}
since
$
\P(\gamma > t | \mathcal F_t^W ) = 1- \P(U \leq \Gamma_t | \mathcal F_t^W )  = 1- \Gamma_t
$ {for $\Gamma=\Gamma(X^D)$}. 
Moreover, since the initial probability of the event $\{\theta=1\}$ is $p$, we also have $\Pi_{0-}:=p$. 
Also note that solving for $\Gamma_t$ in the equation above gives
\begin{equation} \label{Pib}
\Gamma_t =\Gamma_t^\Pi= \frac{p-\Pi_t}{p(1-\Pi_t)},
\end{equation}
so there is a bijection between $\Pi$ and $\Gamma$.

\subsection{Heuristics}
\label{sec:heuristics}

Since 
\[J_1(x,p,D,\Gamma)=\mathbb E \left [ \int_{0}^{\tau_{0} } e^{-rt} (1-p\Gamma_{t-}) d D_t \right]
\leq \mathbb E \left [ \int_{0}^{\tau_0 } e^{-rt}  d D_t \right]\leq V(x)\]
for any strategy pair $(D,\Gamma)\in\mathcal A_1\times\mathcal A_2$,
it is clear that the risk of competition decreases the value from the perspective of Player~1.
On the other hand, to obtain a lower bound, let $\tilde D$ denote the optimal control of the single-player de Finetti problem, see 
\eqref{Dsingle}. Then,
\[J_1(x,p,\tilde D,\Gamma)= \mathbb E \left [ \int_{0}^{\tilde{\tau}_0 } e^{-rt} (1-p\Gamma_{t-}) d\tilde D_t \right]\geq (1-p)\mathbb E \left [ \int_{0}^{\tilde{\tau}_0 } e^{-rt}  d\tilde D_t \right]
= (1-p)V(x)\]
for any $\Gamma\in\mathcal A_2$. It is thus clear that 
\begin{equation}
\label{bound}
(1-p) V(x) \leq J_1(x,p, D^*, \Gamma^*) \leq  V(x)
\end{equation}
if $( D^*, \Gamma^*) \in \mathcal A_1\times \mathcal A_2$ is a Nash equilibrium.

In this section we will provide heuristic arguments to obtain a candidate Nash equilibrium
$(D^*,\Gamma^*)\in \mathcal A_1\times \mathcal A_2$.
To do that, we make the Ansatz that 
\begin{itemize}
\item[(a)]
there exists a non-increasing continuous boundary $p=c(x)$ such that the overall effect of the equilibrium strategy 
$( D^*, \Gamma^*) \in \mathcal A_1\times \mathcal A_2$ amounts to reflection of the two-dimensional process $(X^*,\Pi^*)=(Y-D^*,\Pi^{\Gamma^*})$ along this boundary;
\item[(b)]
the corresponding  equilibrium value $v$ of Player~1 satisfies
\begin{equation}\label{Ansatz}
v(x,p)=(1-p) V(x), \qquad \text{for } p\leq c(x).
\end{equation} 
\end{itemize}
Note that by the bijection between $\Gamma$ and $\Pi$ we have that $\Pi^*=\Pi^*(X^D)$ for every $D\in\cA_1$ and to obtain the reflection of $(X^*,\Pi^*)$ along the monotone boundary $c$ we need that
\begin{equation}\label{AnsatzPi}
\Pi^*_t=\Pi^*_t(X^D)=p\wedge c\Big(\sup_{0\leq s\leq t} X^D_s \Big), \qquad \text{for } t\geq 0.
\end{equation}
With a slight abuse of notation, $\Pi^*$ will be used to indicate both $\Pi^*(X^D)$ and $\Pi^*(X^*)$ but this will be clear from the context as it will depend on whether Player 1 plays an arbitrary admissible strategy $D\in\cA_1$ or the equilibrium strategy $D^*$.

Notice also that the Ansatz \eqref{Ansatz} coincides with the lower bound in \eqref{bound}, and 
is thus of the same type as the equilibrium obtained in the ghost Dynkin game studied in \cite{EkstromDeAngelis}.

Given this Ansatz, we further need to determine 
\begin{itemize}
\item[(i)] 
the boundary $c$;
\item[(ii)]
the direction of reflection when the process $( X^*,\Pi^*)$ is at the boundary;
\item[(iii)]
the strategy pair $(D^*,\Gamma^*)$ corresponding to the reflected process $( X^*,\Pi^*)$.
\item[(iv)]
the strategy for starting points $(x,p)$ with $p>c(x)$;
\end{itemize}
We do this below, and then the candidate Nash equilibrium that we produce is verified in Section~\ref{verification}. Notice that we will not discuss item (iv) here as it is not relevant at this stage, but it will be considered in Theorem \ref{thm:verification}.

First, let us consider a starting point $(x,p)\in[0,\infty)\times(0,1)$ with $p\leq c(x)$,
and recall that we expect in equilibrium that
\[( X^*_t, \Pi^*_t)=\left(Y_t- D^*_t,p\wedge c\Big(\sup_{0\leq s\leq t}(Y_s- D^*_s) \Big)\right),\]
for $D^*\in\mathcal A_1$ to be specified. Since $c$ is assumed to be continuous and non-increasing, we see that 
\begin{equation}
\label{ineq}
p\wedge c\Big(\sup_{0\leq s\leq t}(Y_s- D_s) \Big)\leq c(Y_t-D_t)
\end{equation}
for any choice $D\in\mathcal A_1$. By construction, $\Pi^\ast$ is continuous and we have
\[\Gamma^*_t = \frac{p-\Pi^*_t}{p(1-\Pi^*_t)}\]
and 
\begin{equation}\label{adj}
d \Pi^*_t = -\frac{1}{1-\Gamma^*_t}\Pi^*_t(1-\Pi^*_t) d \Gamma^*_t
\end{equation}
on $\{t\geq 0:\Gamma^*_t<1\}$, 
cf. \eqref{Pi} and \eqref{Pib}.

Note that by the dynamic programming principle one would expect that the process $M=M^D$ given by
$$
M_t := \int_{0}^{t\wedge \hat \gamma^*} e^{-rs} dD_s + e^{-rt}v(X_t, \Pi^*_t)\ind_{\{t<\hat \gamma^*\}}
$$ 
is an $\mathbb F^{W,\hat \gamma^*}$-martingale if $D= D^*\in\mathcal A_1$ is an optimal response to $ \Gamma^* \in\cA_2$, and an
$\mathbb F^{W,\hat \gamma^*}$-supermartingale if $D\in\mathcal A_1$ is any admissible response.
Here, $\mathbb F^{W,\hat \gamma^*}=(\mathcal F^{W,\hat \gamma^*})_{0\leq t<\infty}$
is the smallest right-continuous filtration to which $W$ and $\ind_{\{\cdot \geq \hat\gamma^*\}}$ are adapted, augmented with the $\mathbb P$-null sets of $\Omega$. 
Moreover, by conditioning (cf. Proposition~\ref{alt}), $M$ is an $\mathbb F^{W,\hat \gamma^*}$-(super)martingale if and only if 
\[\hat M_t:= \int_{0}^t e^{-rs}(1-p\Gamma^*_{s-})\,dD_s + e^{-rt}(1-p\Gamma^*_t)v(X_t,\Pi^*_t)\]
is an $\mathbb F^{W}$-(super)martingale.

Thus, by an application of Ito's formula, we see that when Player~2 plays the equilibrium strategy $\Gamma^*$ and $(X^*,\Pi^*)$ is at the boundary we need that
\[(1-v_x)\,dD^*_t-\frac{ \Pi^*_t}{1- \Gamma^*_t}\big((1- \Pi^*_t)v_p\ + v\big)\,d \Gamma^*_t=0 \quad\quad \mbox{(optimality)};\]
{whereas, when Player 2 plays the equilibrium strategy $\Gamma^*$ and Player 1 plays any admissible strategy $D\in\cA_1$, we need that}
\[(1-v_x)\,dD_t-\frac{\Pi^*_t}{1- \Gamma^*_t}\big((1-\Pi^*_t)v_p\ + v\big)\,d \Gamma^*_t\leq 0 \quad\quad \mbox{(suboptimality)},\]
{We stress that $\Pi^*$ here stands for $\Pi^*(X^*)$ in the optimality condition and $\Pi^*(X^D)$ in the suboptimality condition.} Note that we obtain from \eqref{Ansatz} and \eqref{ineq} that
$$
(1-p)v_p(x,p)+v(x,p)=0
$$
{when $p\leq c(x)$}.
Thus, to satisfy the optimality condition we need to have $v_x(x,p)=1$ at the boundary, and consequently the boundary $p=c(x)$ should be defined by 
\[(1-c(x))V'(x)=1\]
for $x\in[0,B]$ where $B$ is as specified in \eqref{eq:B}.
Hence, for $x\in[0,B]$ we have
\begin{equation}\label{Def:c}
c(x)=\frac{V'(x)-1}{V'(x)},
\end{equation}
from which it follows immediately that $c(B)=0$, $c'(x)<0$, and $c'(x) \to 0$ as $x \nearrow B$ by \eqref{eq:Vtilde}.
Let $\hat{p}:=(V'(0)-1)/V'(0)$. Then $c:[0,B]\to[0,\hat{p}]$ is a continuous strictly decreasing bijection and
we denote its inverse by $b:[0,\hat{p}]\to [0,B]$. From here on, we will refer to $b$ (instead of $c$) as the boundary 
when it is more convenient to do so. By convention, we also extend $b$ and $c$ by continuity and define $b(p)=0$ for every $p\in(\hat{p},1]$, and $c(x)=0$ for $x\in(B,\infty)$.

Moreover, notice that since $\Pi^*_t\leq c(X^D_t)$, for every admissible strategy $D\in\cA_1$, we also have that
$$v_x(X^D_t,\Pi^*_t)=(1-\Pi^*_t) {V}'(X^D_t)\geq (1-c(X^D_t)){V}'(X^D_t)=1,$$
so that the suboptimality condition is verified as well.

Since Player 2 in equilibrium only stops at time points when $(X^*,\Pi^*)$ is at the boundary, 
we expect his equilibrium value $u$ to be of the form $u(x,p)=g(p)\psi(x)$, for some function $g$, and to satisfy the condition $u(b(p),p)=b(p)$. Consequently, 
\begin{equation}
\label{u}
u(x,p)= b(p)\frac{\psi(x)}{\psi(b(p)) }
\end{equation}
for $x\leq b(p)$. Furthermore, by the indifference principle for equilibria in randomised strategies,
the process
$$
N_t =  e^{-rt} u(X^*_t, \Pi^*_t)
$$
should be a martingale when Player 1 plays the equilibrium strategy $D^*$. After applying Ito's formula this yields
\begin{equation}
\label{reflectiondirection}
-u_x\,d D^*_t+ u_p\,d \Pi^*_t=0
\end{equation}
on the boundary, so the reflection direction of $(X^*,\Pi^*)$ needs to be $(u_p, -u_x)$. 

We now show how to construct the candidate Nash equilibrium 
$(D^*, \Gamma^*)$ so that the corresponding process $({X^{*}}, \Pi^*)$ reflects along the boundary $c$ in the direction $(u_p, -u_x)$.
To do that, we first specify $\Gamma^*$ by setting {
$$\Gamma^*_t (X^D)= \frac{p-\Pi^*_t}{p(1-\Pi^*_t)}, \quad \text{for } t\geq 0,$$
(cf.~\eqref{Pib}), where $\Pi^*_t=\Pi^*_t(X^D)=p\wedge c(\sup_{0\leq s\leq t}(X^D_s))$ for an arbitrary strategy $D\in\mathcal A_1$.} 
The process {$(X^D,\Pi^*)$} then reflects at the boundary $c$ but the direction of
reflection is, for an arbitrary strategy $D\in\mathcal A_1$, not necessarily equal to $(u_p,-u_x)$.

One should only push in $X=Y-D$ when the process is at its current maximum 
(after the first time it hits the boundary). Therefore, one would expect to choose $ D^*$ so as to satisfy
$$
d D^*_t=\lambda(\bar X^*_t)\,d\bar X^*_t,
$$
where $\bar X^*_t:=b(p)\vee\sup_{0\leq s\leq t} X^*_s$ and $ X^*=Y-D^*$, for some function $\lambda$ to be determined. 
Moreover, from \eqref{AnsatzPi} we have that, when Player 1 plays the equilibrium strategy $D^*$, $ \Pi^*_t=c(\bar X^*_t)$, so \eqref{reflectiondirection}
gives 
\begin{equation}\label{eq:lambda}
\lambda(x)=\frac{c'(x)u_p(x,c(x))}{u_x(x,c(x))}.
\end{equation}
Using \eqref{u}, we then get
\[u_x(x,c(x))= \frac{\psi'(x)}{\psi(x)}x\]
and
\[u_p(x,c(x))=\frac{\psi(x)-x\psi'(x)}{\psi(x)c'(x)},\]
so
\begin{equation}
\label{lambda}
\lambda(x)=\frac{\psi(x)-x\psi'(x)}{x\psi'(x) }.
\end{equation}
and since $\psi$ is concave on $[0,B]$, we have $\lambda\geq 0$ on $(0,B]$.

In the next section we study in detail the solvability of the equation 
$$
X^*_t=Y_t-\int_0^t\lambda(\bar X^*_s)\,d\bar X^*_s
$$
using the notion of {\em perturbed Brownian motion}, which will allow us to obtain the equilibrium strategy $D^*$ for Player 1.

\section{A perturbed Brownian motion with drift}
\label{perturbed}

To construct the equilibrium strategy $D^*$ for Player 1 we will use the notion of perturbed Brownian motion.
Here we provide what is needed for the study of our problem, and refer to \cite{CPY}, \cite{PW} and the references therein 
for further details on such processes. First, define $\Lambda:[b(p),B]\to[0,\infty)$ by
\begin{align}
\label{Lambda}
\Lambda(x)&:=\int_{b(p)}^x\lambda(y)\,dy,
\end{align}
where 
$$
\lambda(x)=\frac{\psi(x)}{x\psi'(x)}-1
$$
as in \eqref{lambda}. Since $\lambda\geq 0$ on $(0,B]$, we note that
$\Lambda$ is increasing. Note also that $\lambda(x)$ is a bounded function {for $x\in[0,B]$} so $\Lambda$ is well-defined. For $x\leq b(p)$ we now consider the equation
\begin{equation}
\label{Xperturbed}
X_t=Y_t- \Lambda(\bar{X}_t), \quad t\in[0,\tau_B],
\end{equation}
where $Y_t=x + \mu t + \sigma W_t $, $\bar{X}_t := b(p) \vee \sup_{0\leq s\leq t} X_s$, and $\tau_B=\tau^X_B:=\inf\{t\geq 0: X_t\geq B \}$.
The process $X$ is then a perturbed Brownian motion with drift. 

To construct a solution of \eqref{Xperturbed}, let 
\begin{equation}\label{eq:Y*}
\bar{Y}_t:=b(p)\vee \sup_{0\leq s\leq t}Y_s.
\end{equation}
Define the function $f:[b(p),\infty)\to[b(p),B]$ by the relations
\begin{align}
\Lambda(f(y))+f(y)&=y, \quad y\in[b(p),\Lambda(B)+B],\label{eq:f}\\
f(y)&=B, \quad y>\Lambda(B)+B, \nonumber
\end{align}
i.e., $f$ is the inverse of the {increasing} function $x\mapsto y:=\Lambda(x) + x$ for $y\in[b(p),\Lambda(B)+B]$ and then extended constantly for $y>\Lambda(B)+B$. Now define
\begin{equation}\label{X}
X_t:=   Y_t-\bar{Y}_t +f(\bar{Y}_t).
\end{equation}
\begin{proposition}\label{Prop:PBm}
Assume that $x\leq b(p)$.
Then the process $X$ in \eqref{X} solves equation \eqref{Xperturbed}.
\end{proposition}

\begin{proof}
Let $t\in[0,\tau_B]$. Since $\bar{X}_t:=b(p)\vee\sup_{s\in[0,t]}X_s$ we obtain, from \eqref{X}, that $\bar{X}_t = f(\bar{Y}_t)$ as $f(b(p))=b(p)$. Consequently $\tau_B=\inf\{t\geq 0: Y_t\geq \Lambda(B)+B \}$ and so, by \eqref{eq:f}, we have
$f(\bar{Y}_t)-\bar{Y}_t= -\Lambda(f(\bar{Y}_t))$. This leads to 
\[X_t=Y_t-\Lambda(\bar{X}_t),\]
which proves the claim.
\end{proof}

\begin{remark}
The set-up in \eqref{Xperturbed} of a perturbed Brownian motion is slightly more general than what is used in 
most literature on perturbed Brownian motions; in fact, the typical choice of perturbation used in the literature is linear,
corresponding to a linear function $\Lambda$ in \eqref{Xperturbed}. On the other hand, we only deal with one-sided perturbation, in which case the solution can be constructed explicitly as in \eqref{X} above. It is straightforward to check that the argument for 
pathwise uniqueness of solutions of \eqref{Xperturbed}, cf. \cite[Proposition 2.1]{CPY}, carries over to our setting.
\end{remark}

\begin{remark}\label{Rmk:f}
	The function $f$ defined in \eqref{eq:f} is constructed in such a way that the process $X_t=Y_t-\bar{Y}_t+f(\bar{Y}_t)$ is a perturbed Brownian motion with drift for $t\in[0,\tau_B]$ (as proved in Proposition \ref{Prop:PBm}) and it is the Skorokhod reflection of the process $Y_t$ at the barrier $B$ for $t\in(\tau_B,\infty)$. Indeed, for $t\in(\tau_B,\infty)$, we have
	\begin{align}\label{Xreflect}
	X_t&=Y_t-\bar{Y}_t+f(\bar{Y}_t)=Y_t-\bar{Y}_t+B\nonumber\\
	&=Y_t-\sup_{s\in[0,t]}(Y_s-B)=Y_t-\sup_{s\in[0,t]}(Y_s-B)^+,
	\end{align}
	i.e., we have $X_t=X^{\tilde{D}}_t$ for $t\in(\tau_B,\infty)$ where $\tilde{D}$ is defined as in \eqref{Dsingle}.
\end{remark}

\section{Main result}
\label{verification}

In this section, we state and prove our main result: an explicit Nash equilibrium for our game. To do that, let us fix $(x,p)\in[0,\infty)\times[0,1]$ and recall that $Y$ is given by
\[Y_t=x+ \mu t+\sigma W_t.\]
First, define a new process $ Y^\wedge$ by 
\[ Y^\wedge_t:=x\wedge b(p)+ \mu t+\sigma W_t= Y_t-(x-b(p))^+,\]
so that $Y^\wedge$ starts below the boundary $b(p)$.
Then define $\bar Y^\wedge$ as in \eqref{eq:Y*} but with $Y^\wedge$ instead of $Y$, i.e.,
\[\bar Y^\wedge_t:=b(p)\vee \sup_{0\leq s\leq t} Y^\wedge_s.\]
Also, recall the definitions of $\Lambda:[b(p),B]\to[0,\infty)$ in \eqref{Lambda} and $f:[b(p),\infty)\to [b(p),B]$ in \eqref{eq:f}, and define $ D^*\in\mathcal A_1$ by $D^*_{0-}=0$ and
\begin{equation}
\label{D*}
D^*_t:= (x-b(p))^+ +\bar Y^\wedge_t -f(\bar Y^\wedge_t), \quad t\geq 0.
\end{equation}
Setting 
\[X^{\ast}_t:= Y_t-D^*_t,\]
Proposition~\ref{Prop:PBm} applied with $Y^\wedge$ in place of $Y$ yields
\begin{equation}\label{X^D*}
X^{\ast}_t=Y^\wedge_t-\bar{Y}^\wedge_t+f(\bar{Y}^\wedge_t)=Y^\wedge_t-\Lambda(\bar{X}^{\ast}_t), \qquad t\in[0,\tau^{\ast}_B],
\end{equation}
where $\tau^\ast_B=\tau_B^{X^{\ast}}:=\inf\{t\geq 0: X^{\ast}_t\geq B \}$. Note that by construction we have $dD^\ast_t = \Lambda (X_t^\ast) d\bar X_t^\ast$ for $t \in (0,\tau_B]$.

Moreover, for a given path $X=X^D\in\mathcal D$ (with $D \in \mathcal A_1$), define $Z^*=Z^*(X)$ by $Z^*_{0-}:=p$ and 
\begin{equation}\label{eq:Z}
Z_t^*:=p\wedge c\Big(\sup_{0\leq s\leq t}X_s\Big), \quad t\geq 0
\end{equation}
(cf. \eqref{AnsatzPi}),
and define $\Gamma^*\in\mathcal A_2$ by
\begin{equation}\label{Gamma*}
\Gamma^*_t(X):=\left\{\begin{array}{ll} 
\ind_{\{t\geq \tau_B\}}, & p=0,\\
\frac{p-Z^*_t}{p(1-Z^*_t)}, & p>0,
\end{array}
\right.
\end{equation}
where we recall that $\tau_B:=\inf\{t\geq 0:X_t\geq B\}$.

\begin{theorem}\label{thm:verification}
Let $(x,p)\in[0,\infty)\times [0,1]$. 
The pair $( D^*,\Gamma^*)$ defined above is a NE for the stochastic game \eqref{eq:valuefunctions}-\eqref{eq:valuefunctions2}, with equilibrium values
\begin{align*}
 J_1(x,p, D^*, \Gamma^*) &= v(x,p) :=
\left\{\begin{array}{ll} (1-p) V(x), & x\leq b(p),\\
(1-p)V(b(p)) + x-b(p), &x>b(p),\end{array}\right. \\
J_2(x,p, D^*, \Gamma^*) &= u(x,p) :=\left\{\begin{array}{ll} b(p)\frac{\psi(x)}{\psi(b(p)) }, &x\leq b(p),\\
b(p), &x>b(p),\end{array}\right.
\end{align*}
(with the understanding that $b(p)\psi(x)/\psi(b(p))=0$ for $x=0$ also when $b(p)=0$).
Here, $V$ is the value of the single-player de Finetti problem given in \eqref{eq:Vtilde}, and $\psi$ is given by \eqref{eq:psi}. 
\end{theorem}
\begin{proof}
\textit{Step 1}. We first prove that $D^*$ is an optimal response to $\Gamma^*$. Let $D\in\mathcal A_1$ be an arbitrary strategy for Player 1 and set $X:=Y-D$. Let $Z^*$ be defined as in \eqref{eq:Z} and $ \Gamma^*_t:= \Gamma^*_t(X)$ as in \eqref{Gamma*} accordingly .

If $p=0$, then $\theta=0$ a.s.~and so
$$J_1(x,0,D,\Gamma^*)=\bE\bigg[\int_{0}^{\tau_0}e^{-rt}dD_t\bigg].$$
Namely, the optimization problem for Player 1 degenerates into the single-player de Finetti problem, and $D^*$ coincides with its optimal solution $\tilde{D}$, as highlighted in Remark~\ref{Rmk:f}. Hence, also $ v(x,0)=J_1(x,0,D^*,\Gamma^*)\geq J_1(x,0,D,\Gamma^*)$ for every $D\in\cA_1$.

If $x=0$, then $J_1(0,p,D,\Gamma^*)=0$ for every $p\in[0,1]$, $D\in\cA_1$ and so, in particular, $v(0,p)=J_1(0,p,D^*,\Gamma^*)$ for every $p\in[0,1]$.

Now let $p\in(0,1]$ and let us first consider $0<x\leq b(p)$ (note that this implies that $p\in(0,\hat{p})$ as $b(p)=0$ for every $p\in[\hat{p},1]$). By \eqref{Gamma*}, we have 
$$Z^*_t=\frac{p(1-\Gamma^*_t)}{1-p \Gamma^*_t},\quad t\geq 0.$$
Since $Z^*$ and $\Gamma^*$ are continuous and of finite variation, 
we obtain
$$
dZ^*_t=-\frac{p(1-Z^*_{t})}{1-p\Gamma^*_{t}}\,d \Gamma^*_t,\quad t\geq 0.
$$

Now define
$$\tilde{v}(x,p):=(1-p)V(x)\in C^2([0,\infty)\times[0,1]).$$

By setting $\tau:=\tau_0\wedge T$ with $T\geq 0$ and applying Ito's formula to $e^{-rt}(1-p\Gamma^*_{t}) \tilde{v}(X_t,Z^*_t)$, we have that
\begin{align}\label{ItoD*}
e^{-r\tau}(1-p\Gamma^*_\tau) \tilde{v}(X_\tau, Z^*_\tau)&= \tilde{v}(x,p)+\int_0^\tau e^{-rt}(1-p \Gamma^*_{t})\cL \tilde{v}(X_{t-},Z^*_{t})\,d t\nonumber\\
&\hspace{12pt}-\int_0^\tau e^{-rt}(1-p\Gamma^*_{t})\tilde{v}_x(X_{t-},Z^*_{t})\,d  D^c_t\nonumber\\
&\hspace{12pt}+\int_0^\tau \sigma e^{-rt}(1-p \Gamma^*_{t})  \tilde{v}_x(X_{t-},Z^*_{t})\,d W_t\nonumber\\
&\hspace{12pt}-\int_0^\tau e^{-rt}p\big[(1-Z^*_{t}) \tilde{v}_p(X_{t-},Z^*_{t})+ \tilde{v}(X_{t-},Z^*_{t})\big]d\Gamma^{*}_t\nonumber\\
&\hspace{12pt}+\sum_{0\leq t \leq \tau} e^{-rt}(1-p \Gamma^*_{t})\big( \tilde{v}(X_t,Z^*_t)- \tilde{v}(X_{t-},Z^*_{t})\big),
\end{align}
where $\cL$ is defined as in \eqref{eq:L} and $D^c$ denotes the continuous part of $D$. Notice that $\tilde{v}(x,p)=v(x,p)$ {for $x\leq b(p)$} and that by definition of $\tilde{v}$, we have for every $t>0$
$$\cL \tilde{v}(X_{t-},Z^*_{t})= 0 \quad \text{and} \quad (1-Z^*_{t}) \tilde{v}_p(X_{t-},Z^*_{t})+ \tilde{v}(X_{t-},Z^*_{t})=0.$$
Hence, equation \eqref{ItoD*} becomes
\begin{align}\label{ItoD*2}
v(x,p)&= e^{-r\tau}(1-p\Gamma^*_\tau) \tilde{v}(X_\tau, Z^*_\tau)+\int_0^\tau e^{-rt}(1-p\Gamma^*_{t})  \tilde{v}_x(X_{t-},Z^*_{t})\,d D^c_t\nonumber\\
&\hspace{12pt}-\int_0^\tau \sigma e^{-rt}(1-p\Gamma^*_{t})\tilde{v}_x(X_{t-},Z^*_{t})\,d W_t\nonumber\\
&\hspace{12pt}-\sum_{0\leq t \leq \tau} e^{-rt}(1-p\Gamma^*_t)\big( \tilde{v}(X_t,Z^*_t)-\tilde{v}(X_{t-},Z^*_{t})\big).
\end{align}
For the summation term we have by the mean value theorem that
\begin{equation}\label{eq:Jumps}
\sum_{0\leq t \leq \tau} e^{-rt}(1-p\Gamma^*_t)\big( \tilde{v}(X_t,Z^*_t)-\tilde{v}(X_{t-},Z^*_{t})\big)=-\sum_{0\leq t \leq \tau} e^{-rt}(1-p\Gamma^*_{t}) \tilde{v}_x(\xi_t,Z^*_{t})\Delta D_t
\end{equation}
where $\xi_t\in(X_{t-},X_{t})$ and $\Delta D_t:=D_t-D_{t-}$.
By plugging \eqref{eq:Jumps} into \eqref{ItoD*2}, and using that $ \tilde{v}\geq 0$ and $ \tilde{v}_x\geq 1$, we obtain
\begin{equation}\label{ItoD*4}
 v(x,p)\geq \int_{0}^\tau e^{-rt}(1-p\Gamma^*_{t}) \, dD_t-\int_0^\tau \sigma e^{-rt}(1-p\Gamma^*_{t}) \tilde{v}_x(X_{t-},Z^*_{t})\, d W_t.
\end{equation}
Let
\begin{equation}\label{eq:defA}
\mathcal O := \{(x,p)\in[0,\infty)\times[0,1]:x\leq b(p) \}\cup((B,\infty)\times\{0\})
\end{equation}
and note that $(X_{t-},Z^*_t)\in \mathcal O$ for every $t\geq 0$ (by construction of $Z_t$) and that $\tilde{v}_x$ is bounded on $\mathcal O$ ($ \tilde{v}_x(x,p)=1$ for $(x,p)\in(B,\infty)\times\{0\}$). Thus, the stochastic integral above is a martingale and by an application of the optional sampling theorem we have that
$$ \tilde{v}(x,p)\geq \bE\bigg[\int_{0}^{\tau_0\wedge T} e^{-rt}(1-p \Gamma^*_{t})\,d D_t \bigg].$$
Letting $T\to\infty$ yields, by the monotone convergence theorem,
\[v(x,p)\geq \bE\bigg[\int_{0}^{\tau_0} e^{-rt}(1-p \Gamma^*_{t})\,d D_t \bigg]=\bE\bigg[\int_{0}^{\tau_0} e^{-rt}(1-p \Gamma^*_{t-})\,d D_t \bigg]=J_1(x,p,D, \Gamma^*)\]
for every $D\in\cD$, where the last equality follows by Proposition~\ref{alt}.

Now notice that $D^*_t$ defined in \eqref{D*} is continuous for every $t\geq0$, when $x\leq b(p)$, and that the same holds for $X^*_t:=X^{ D^*}_t$. 
Let $\tau^*_0:=\tau^{X^*}_0$, then equation \eqref{ItoD*2} for $D=D^*$ and $\tau^*:=\tau^{*}_0\wedge T$ becomes
\begin{eqnarray*}
v(x,p) &=& e^{-r\tau^*}(1-p\Gamma^*_{\tau^*}) \tilde{v}(X^*_{\tau^*}, Z^*_{\tau^*})+ \int_0^{\tau^{*}} e^{-rt} (1-p \Gamma^*_{t})  \tilde{v}_x(X^{*}_{t},Z^*_{t})\,dD^{*}_t\\
&& -\int_0^{\tau^{*}} \sigma e^{-rt}(1-p\Gamma^*_{t})  \tilde{v}_x(X^{*}_{t},Z^*_{t})\,d W_t\\
&=&e^{-r\tau^*}(1-p\Gamma^*_{\tau^*}) \tilde{v}(X^*_{\tau^*}, Z^*_{\tau^*})+\int_{0}^{\tau^{*}} e^{-rt}(1-p\Gamma^*_{t})\,dD^*_t\\
&&-\int_0^{\tau^{*}} \sigma e^{-rt}(1-p\Gamma^*_{t}) v_x(X^{*}_{t},Z^*_{t})\,d W_t,
\end{eqnarray*}
where the last equality holds since $\tilde{v}_x(x,p)=1$ if $x\geq b(p)$ and $dD^*_t=0$ if $X^*_t< b(Z^*_t)$.
Hence, again by taking expected values, we obtain
\begin{eqnarray*}
v(x,p)&=& \bE\bigg[e^{-r(\tau^{*}_0\wedge T)}(1-p\Gamma^*_{\tau^{*}_0\wedge T}) \tilde{v}(X^*_{\tau^{*}_0\wedge T}, Z^*_{\tau^{*}_0\wedge T})+\int_{0}^{\tau^{*}_0\wedge T} e^{-rt}(1-p \Gamma^*_{t})\,d D^*_t \bigg]\\
&\to& \bE\bigg[\int_{0}^{\tau^{*}_0} e^{-rt}(1-p \Gamma^*_{t})\,d D^*_t \bigg]
\end{eqnarray*}
as $T\to\infty$ by dominated convergence (the first term tends to 0 since $\tilde{v}(X^*_{\tau^{*}_0}, Z^*_{\tau^{*}_0})=0$). Thus, we have proved that
$$
J_1(x,p, D^*, \Gamma^*)=v(x,p)\geq \sup_{D\in\mathcal A_1}J_1(x,p,D,\Gamma^*), \quad \forall \: (x,p)\in \mathcal O.
$$

Let us now consider $(x,p)\in ([0,\infty)\times[0,1])\setminus  \mathcal O=:\mathcal O^c$, i.e., $x>b(p)$ with $p\neq 0$. Then,
$$v(x,p)=v(b(p),p)+x-b(p)=J_1(b(p),p,D^*,\Gamma^*)+x-b(p)=J_1(x,p,D^*,\Gamma^*).$$
Thus, we are left to prove that also in this case 
$$J_1(x,p,D^*,\Gamma^*)\geq J_1(x,p,D,\Gamma^*), \quad \forall \: D\in\cA_1.$$
For $(x,p)\in \
\mathcal O^c$, let the admissible strategy $D\in\cA_1$ have an initial jump $\Delta D_0=x-y$ with either $b(p)\leq y\leq x$ or $0\leq y< b(p)$. In the former case, by definition \eqref{Gamma*} of $\Gamma^*$, we have that
$$J_1(x,p,D,\Gamma^*)=(1-\Gamma^*_0)J_1(b(q),q,D,\Gamma^*)+x-y = \frac{q(1-p)}{p}V(b(q))+x-y,$$
where $q:=c(y)\leq p$ (and hence $y=b(q)$). Since $V$ is concave with $V'(b(p))=1/(1-p)$, then
\begin{align*}
J_1(x,p,D,\Gamma^*)&\leq \frac{q(1-p)}{p}\Big(V(b(p))+\frac{y-b(p)}{1-p}\Big)+x-y\\
&=\frac q p \Big((1-p)V(b(p))+y-b(p)\Big)+x-y\\
&\leq (1-p)V(b(p))+x-b(p)=J_1(x,p,D^*,\Gamma^*).
\end{align*}
If instead $0\leq y< b(p)$, then by a similar argument
\begin{align*}
J_1(x,p,D,\Gamma^*)&=J_1(y,p,D,\Gamma^*)+x-y=(1-p)V(y)+x-y\\
&\leq (1-p)V(b(p))+x-b(p)=J_1(x,p,D^*,\Gamma^*).
\end{align*}
This concludes Step 1, i.e., shows that the strategy $D^*$ is an optimal response to $\Gamma^*$.

\textit{Step 2}. We now prove that $\Gamma^*$ is an optimal response to $D^*$. Recall that
$$u(x,p):=\left\{\begin{array}{ll} b(p)\frac{\psi(x)}{\psi(b(p)) }, & x\leq b(p),\\
b(p), &x>b(p),\end{array}\right.$$
set $X^*:=X^{D^*}$ with $D^*$ defined in \eqref{D*}, $ \tau^*_0:=\tau^{X^*}_0$, and let 
\begin{equation*}
Z^*_t:=p\wedge c\Big(\sup_{0\leq s\leq t} X^*_s\Big), \quad t\geq 0, \qquad Z^*_{0-}:=p,
\end{equation*}
as in \eqref{eq:Z} with $D=D^*$.

Let $p\in[0,1]$ and assume $x\leq b(p)$. If $p\in[\hat{p},1]$, then $b(p)=0$ and so $x=0$ and the strategy $\Gamma\in\cA_2$ is irrelevant since the game stops immediately. It hence suffices to check $p \in [0, \hat p)$.  For notational convenience we treat the case $p=0$ separately at the end and assume first $p \in (0,\hat p)$. Note that $X^*_t\leq b(Z^*_t)$ for every $t\geq 0$ and that $Z^*_t$, $D^*_t$ and $X^*_t$ are continuous for every $t\geq 0$. Define 
$$
\tilde{u}(x,p):=b(p)\frac{\psi(x)}{\psi(b(p))}\in C^2([0,\infty)\times(0,\hat{p})).
$$
and let $\tau$ be any $\mathbb F^W$-stopping time {s.t.~$\tau \leq \tau^\ast_B$ $a.s.$, where $\tau^\ast_B=\inf\{t\geq 0: X^\ast_t \geq B\}$. Define $\tau^*=\tau^*_{\varepsilon,T}:=\tau^*_0\wedge \tau^\ast_{B-\varepsilon}\wedge\tau\wedge T$ for $T,\varepsilon\geq 0$ arbitrary and note that $Z^\ast_t >0$ for $t \in [0,\tau^\ast)$.} By applying Ito's formula to $e^{-rt} \tilde{u}(X^*_t,Z^*_t)$ we obtain
\begin{align}
e^{-r\tau^*} \tilde{u}(X^*_{\tau^*},Z^*_{\tau^*}) &= \tilde{u}(x,p)+\int_0^{\tau^*} e^{-rs}\cL  \tilde{u}(X^*_s,Z^*_s)\,ds-\int_0^{\tau^*} e^{-rs} \tilde{u}_x(X^*_s,Z^*_t)\,d D^{*}_s\nonumber\\
&\hspace{12pt}+\int_0^{\tau^*} \sigma e^{-rs} \tilde{u}_x(X^*_{s},Z^*_{s})\,dW_s+\int_0^{\tau^*} e^{-rs} \tilde{u}_p(X^*_s,Z^*_s)\,dZ^{*}_s.\nonumber
\end{align}
By definition of $\tilde{u}$, we have that $\cL  \tilde{u}(X^*_s,Z^*_s)=0$ for every $0\leq s\leq \tau^*$
and by construction of $D^*$ and $Z^*$ (recall \eqref{X^D*}), we obtain
\begin{align}\label{ZD*ineq}
\int_0^{\tau^*}& e^{-rs} \tilde{u}_p(X^*_s,Z^*_s)\,dZ^{*}_s-\int_0^{\tau^*} e^{-rs}\tilde{u}_x(X^*_s,Z^*_s)\,dD^{*}_s \nonumber\\
&=\int_0^{\tau^*} e^{-rs}\Big(\tilde{u}_p(X^*_s,Z^*_s)c'(X^*_s)-\tilde{u}_x(X^*_s,Z^*_s)\lambda(X^*_s) \Big)\,d\bar{X}^{*}_s =0 \nonumber\\
\end{align}
where the last equality holds by definition of $\lambda$ in \eqref{eq:lambda}.  
Hence,
\begin{equation}
\label{ItoU}
e^{-r\tau^*} \tilde{u}(X^*_{\tau^*},Z^*_{\tau^*}) {=} \tilde{u}(x,p)+\int_0^{\tau^*} \sigma e^{-rs}  \tilde{u}_x(X^*_{s},Z^*_{s})\,dW_s.
\end{equation}
Since $\tilde{u}_x$ is bounded on $\{(x,p):x\leq b(p)\}$, the stochastic integral in \eqref{ItoU} is a martingale. Since $X^\ast$ and $Z^\ast$ are continuous, applying the optional sampling theorem and using dominated convergence yields
\begin{equation*}
\tilde{u}(x,p) = \bE\bigg[e^{-r\tau^*} \tilde{u}(X^*_{\tau^*},Z^*_{\tau^*})\bigg]\to\bE\bigg[e^{-r(\tau^*_0\wedge \tau)} \tilde{u}(X^*_{\tau^*_0\wedge \tau},Z^*_{\tau^*_0 \wedge \tau})\bigg],
\end{equation*}
as $T\to\infty$ and $\varepsilon \to 0$, so
\begin{equation}\label{ItoU1}
\tilde{u}(x,p) =\bE\bigg[e^{-r(\tau^*_0\wedge \tau)} \tilde{u}(X^*_{\tau^*_0\wedge \tau},Z^*_{\tau^*_0 \wedge \tau})\bigg]
\end{equation}
for any $\bF^W$-stopping time $\tau \leq \tau_B$ $a.s.$
Now, for any $\Gamma\in\mathcal A_2$, define the $\bF^W$-stopping times
$$
\gamma(\rho):=\inf\{t\geq 0:\Gamma_t( X^*)> \rho\},\quad\quad \rho\in[0,1),
$$
{and let $\gamma_B(\rho) := \gamma({\rho}) \wedge \tau^\ast_B \leq \tau^\ast_B$.}
Since $\tilde{u}=u$ on $\{(x,p):x\leq b(p)\}$, {equality \eqref{ItoU1} }for $\tau=\gamma_B(\rho)$ reads
$$
u(x,p)=\bE\bigg[e^{-r(\tau^*_0\wedge \gamma_B(\rho))} u( X^*_{\tau^*_0\wedge \gamma_B(\rho)},Z^*_{ \tau^*_0\wedge \gamma_B(\rho)})\bigg],\quad\quad \rho\in[0,1).
$$
Thus,
\begin{align}\label{ItoU2}
u(x,p)&=\int_0^1 \bE\bigg[e^{-r(\tau^*_0\wedge \gamma_B(\rho))} u( X^*_{\tau^*_0\wedge \gamma_B(\rho)},Z^*_{\tau^*_0\wedge \gamma_B(\rho)})\bigg]\,d\rho\\
&\geq \int_0^1\bE\bigg[e^{-r(\tau^*_0 \wedge \gamma_B(\rho))}X^*_{\tau^*_0\wedge \gamma_B(\rho)}\bigg]\,d\rho\nonumber
\end{align}
where the inequality holds because $\psi(x)$ is concave for $x\leq B$ with $\psi(0)=0$.

Last, we note that 
\begin{equation} \label{eq:equality}
e^{-r(\tau^*_0\wedge \gamma_B(\rho))}X^*_{\tau^*_0\wedge \gamma_B(\rho)} \geq e^{-r(\tau^*_0\wedge \gamma(\rho))}X^*_{\tau^*_0 \wedge \gamma(\rho)} \qquad  a.s.
\end{equation}
since $X_t^\ast \leq B$ for all $t >0$ and $r>0$ and thus
\begin{equation*}
u(x,p) \geq  \int_0^1 \mathbb E \bigg[e^{-r(\tau^*_0\wedge \gamma(\rho))}X^*_{\tau^*_0\wedge \gamma(\rho)} \bigg] d\rho = J(x,p, D^\ast, \Gamma).
\end{equation*}
{If $\Gamma= \Gamma^*$, then by \eqref{Gamma*} we have that $\gamma^*(\rho)\leq \tau^*_B$ for every $\rho\in[0,1)$, where
$$\gamma^*(\rho):=\inf\{t\geq 0:\Gamma^*_t( X^*)> \rho\},\quad\quad \rho\in[0,1)$$ and thus the inequality in \eqref{eq:equality} is an equality in this case. 
Moreover, $\Gamma_t^\ast$ only increases when $Z_t^\ast$ increases and $Z^\ast = Z_t^*:=p\wedge c\Big(\sup_{0\leq s\leq t}X_s\Big)$ so 
$$
u( X^*_{\tau^*_0 \wedge\gamma^\ast(\rho)},Z^*_{\tau^*_0 \wedge\gamma^\ast(\rho)}) = b(c(\bar X_{\tau^*_0 \wedge\gamma^\ast(\rho)}))  = X^\ast_{\tau^*_0 \wedge\gamma^\ast(\rho)}
$$
in \eqref{ItoU2}. Thus all the inequalities above become equalities and
\begin{equation}\label{OptimGamma^*}
u(x,p)= J_2(x,p, D^*,\Gamma^*).
\end{equation}
}

If $p=0$, we have $u(x,0)=\tilde u(x,0) =b(0) \frac{\psi(x)}{\psi(b(0))} = B \frac{\psi(x)}{\psi(B)}$ and $Z^\ast_t=0$ for all $t\geq 0$. Applying Ito's formula to $e^{-rt}u(X_t^\ast,0)$ between $0$ and $\tau_0 \wedge \tau \leq \tau
^\ast_B$ and using the properties of $\psi(x)$ gives
\begin{align*}
e^{-r(\tau_0 \wedge \tau)} \tilde u(X_{\tau_0 \wedge \tau}, 0) =\tilde u(x,0)   - \int_0^{\tau_0 \wedge \tau}  e^{-rs} \tilde u_x(X_s^\ast,0)dD^\ast_s + \int_0^{\tau_0 \wedge \tau}  e^{-rs} \sigma \tilde u_x(X^\ast_s,0) dW_s.
\end{align*}
Taking expected value and arguing as above thus gives
$$
u(x,0)= \mathbb E\bigg[e^{-r(\tau_0 \wedge \tau^\ast_B)} u(X^\ast_{\tau_0 \wedge \tau^\ast_B},0)\bigg] =  e^{-r(\tau_0 \wedge \tau^\ast_B)} X_{\tau_0 \wedge \tau^\ast_B} = J_2(x,0,D^\ast, \Gamma^\ast)
$$
and 
\begin{align*}
u(x,0) =& \int_0^1 \mathbb E\bigg[e^{-r(\tau_0 \wedge{\gamma_B(\rho)})} u(X_{\tau_0 \wedge\gamma_B(\rho)},0)\bigg]  d\rho \geq  \int_0^1 \mathbb E\bigg[e^{-r{(\tau_0 \wedge\gamma_B(\rho))}} X_{\tau_0 \wedge\gamma_B(\rho)}\bigg]d\rho \\ \geq& J_2(x,p,D^\ast, \Gamma)
\end{align*}
where we again have used convexity of $\psi$ and the fact that any stopping time $\gamma(\rho)>\tau^\ast_B$ yields a lower payoff that $\tau^\ast_B$.

The above treats the case $x\leq b(p)$ so let us finalize the proof by considering $x>b(p)$. We have, for every $\Gamma\in\cA_2$, that
$$u(x,p)=u(b(p),p)\geq J_2(b(p),p,D^*,\Gamma)=J_2(x,p,D^*,\Gamma),$$
where the last equality holds by the precedence of Player 1 over Player 2 and since $D^*_0=x-b(p)$ for $x>b(p)$. Similarly, we obtain
$$u(x,p)=u(b(p),p)= J_2(b(p),p,D^*,\Gamma^\ast)=J_2(x,p,D^*,\Gamma^\ast).$$ 

Hence, $ \Gamma^*$ is an optimal response to $D^*$. Together with Step 1, this implies that $(D^*,\Gamma^*)$ is a NE and that the equilibrium values are $v$ and $u$, respectively. This concludes the proof.
\end{proof}

\begin{remark} \label{remark:pi0}
It is a remarkable feature of the equilibrium stratgey $(D^\ast, \Gamma^\ast)$ that it allows the process $\Pi^\ast$ to reach $0$ in finite time, thereby completely ruling out the possibility that a fraudster exists if he did not stop the game yet. 
Indeed, let $x\leq b(p)$, then we have
$$
X^\ast_t=   Y_t-\bar{Y}_t +f(\bar{Y}_t)
$$
and thus $\bar X^\ast_t = f(\bar Y_t)$ where $f$ is an increasing bounded function such that $f(x) =B$ for all $x\geq \Lambda(B)+B$. Consequently, $\Pi^\ast_{t} = p \wedge c(\bar X^\ast_t) =p \wedge c(f(\bar Y_t))=c(B)=0$ for all 
$$t \geq  \tau_B=\inf \{s \geq 0: Y_s \geq \Lambda(B)+B \},$$ 
the first time the unrestricted Brownian motion (with drift) $Y$ reaches $\Lambda(B) +B$ (which is finite a.s.). 
\end{remark}

\section{A numerical example}
\label{num}

To provide the reader with further intuition, we conclude by looking at some numerical experiments. Throughout the section, we consider parameters $\mu=0.03$, $\sigma=0.12$, and $r=0.01$. The optimal strategy $\tilde D$ in the single-player {de Finetti} problem given by \eqref{Dsingle} then amounts to reflection at $B\approx 1.12$. 

Note that whereas the qualitative form of the single-player strategy {de Finetti} problem is fixed, the nature of the NE strategy for Player 1 varies depending on the value of $p\in[0,1]$. To be more precise, if Player 1 is certain that no fraudster exists, i.e., if $p=0$, then the problem degenerates into the standard single-player {de Finetti} problem and the optimal strategy is $\tilde D$ (and Player 2 would stop as soon as $X$ hits $B$). On the other hand, if Player 1 has sufficient evidence of the existence of a fraudster, i.e., if $p\in [\hat{p},1]$ where $\hat{p}:=(V'(0)-1)/V'(0)$, then the agent extracts the whole resource immediately and the game terminates at $t=0$. 
The most interesting scenario is when $p\in(0,\hat{p})$. In this case, the NE described in Theorem \ref{thm:verification} amounts to a {(possible)} initial lump sum extraction of size {$(x-b(p))^+$}, and then continuous 
extraction as to reflect the two-dimensional process $(X^\ast,\Pi^\ast)$ along the boundary $b$, with reflection
in the prescribed direction {$(u_p,-u_x)$}. Figures \ref{fig:path6390} and \ref{fig:processes6390} are derived with initial values $p_0= 0.8 \cdot \hat p \approx 0.72$ and $x_0 = \frac{b(p_0)}{2} \approx 0.13$, putting us in the last of the three cases above.  

Figure \ref{fig:boundary} shows the boundary $p\mapsto b(p)$ (or equivalently $x\mapsto c(x)$) together with the direction of reflection of the equilibrium process $(X^\ast, \Pi^\ast)$. Note that $b(0)=B$ and $b(\hat p)=0$. Figures \ref{fig:path6390} and \ref{fig:processes6390} show a simulated path of the equilibrium process $(X^\ast, \Pi^\ast)$ and the corresponding processes $\Pi^\ast$, $\Gamma^\ast$, and $D^\ast$, respectively. Flat portions of $\Gamma^\ast, \Pi^\ast$, and $D^\ast$ correspond to $X^\ast$ being strictly below the boundary $b(\Pi^\ast)$. Note also that in Figure \ref{fig:path6390}, the process $\Pi^\ast$ reaches $0$ in finite time, ruling out the existence of a fraudster playing the equilibrium strategy if he did not stop yet, see Remark \ref{remark:pi0}. 

\begin{figure}
\centering
\includegraphics[width=0.85\textwidth,height=0.85\textheight,keepaspectratio]{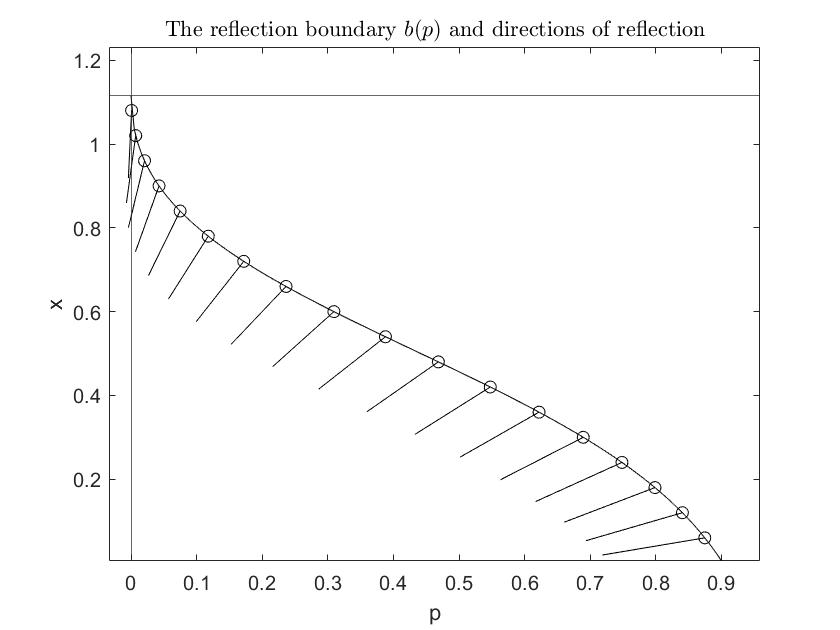}
\caption{The boundary $b(p)$ and the direction of reflection for the equilibrium process $(X^\ast, \Pi^\ast)$.
}
\label{fig:boundary}
\end{figure}

\begin{figure}
\centering
\includegraphics[width=0.85\textwidth,height=0.85\textheight,keepaspectratio]{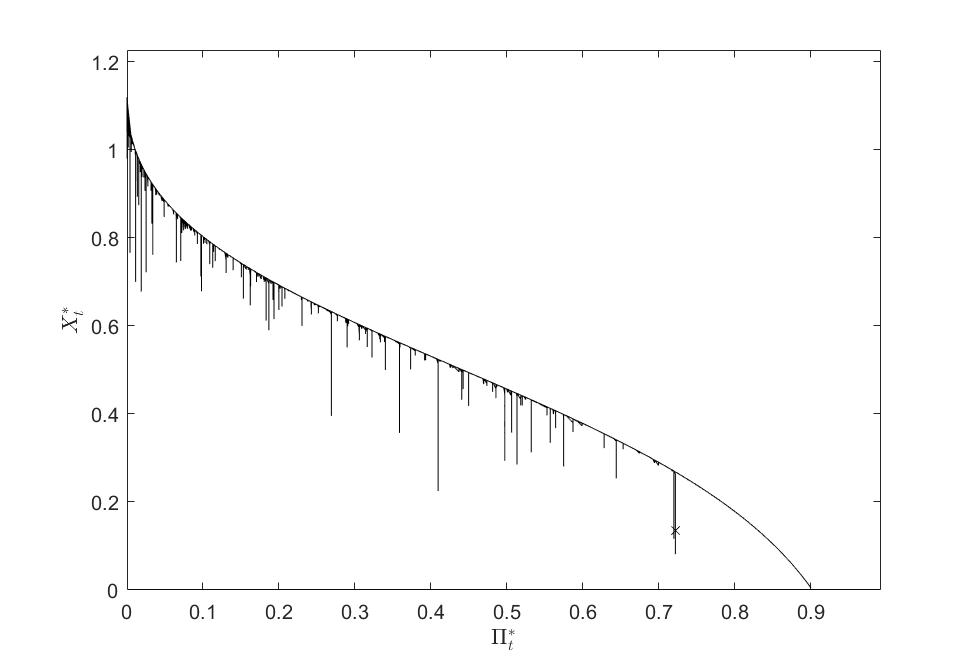}
\caption{A simulated path of $(\Pi^\ast, X^\ast)$ {reflected along the boundary $p\mapsto b(p)$}.}
\label{fig:path6390}
\end{figure}

\begin{figure}
\centering
\includegraphics[width=0.85\textwidth,height=0.85\textheight,keepaspectratio]{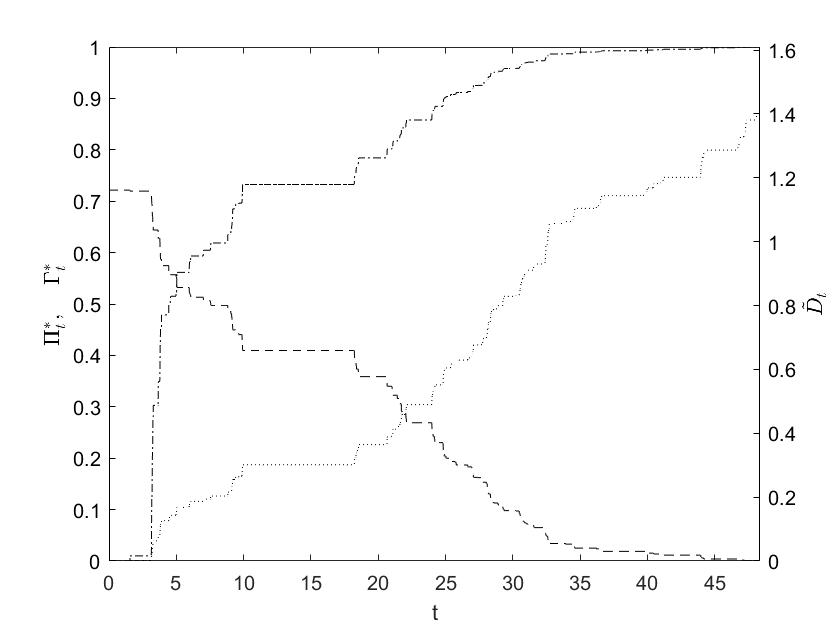}
\caption{Auxiliary processes $\Pi^\ast$ (dashed), $\Gamma^\ast$ (dash-dot), and $D^\ast$ (dotted).}
\label{fig:processes6390}
\end{figure}

 \newpage
\bibliography{bibfile}{}
\bibliographystyle{abbrv}

\end{document}